\documentclass[12pt]{amsart}

\usepackage{hyperref}

\usepackage{enumerate}
 
\usepackage{graphicx}
\usepackage{xcolor}
 
\usepackage{amssymb}

\newtheorem{teo}{Theorem}[section]
\newtheorem{exer}[teo]{\sffamily\bfseries Ejercicio}
\newtheorem{lem}[teo]{Lemma}
\newtheorem{cor}[teo]{Corollary}
\newtheorem{prop}[teo]{Proposition}
 
\newtheorem{defi}[teo]{Definition}
\newtheorem{rem}[teo]{Remark}

\numberwithin{equation}{section}

\newcommand{\dist}{\operatorname{dist}}
\newcommand{\diag}{\operatorname{Diag}}
\newcommand{\kah}{\mathcal{K}(\mathcal{H})^{ah}}
\newcommand{\dbh}{\mathcal{D}\left(\mathcal{B}\left(\mathcal{H}\right)\right)}
\newcommand{\dkh}{\mathcal{D}\left(\mathcal{K}\left(\mathcal{H}\right)\right)}
\newcommand{\dah}{\mathcal{D}\left(\mathcal{B}\left(\mathcal{H}\right)\right)^{ah}}
\newcommand{\kh}{\mathcal{K}(\mathcal{H})}
\newcommand{\uh}{\mathcal{U}(\mathcal{H})}

\newcommand{\bh}{\mathcal{B}(\mathcal{H})^{h}}
\newcommand{\bah}{\mathcal{B}(\mathcal{H})^{ah}}
\newcommand{\ob}{\mathcal{O}_b}
\newcommand{\oa}{\mathcal{O}_A}
\newcommand{\oo}{\mathcal{O}}
\newcommand{\uu}{\mathcal{U}}
\newcommand{\ukd}{{\mathcal {U}}_{k,d}}
\newcommand{\uda}{\mathcal{U}_{k+d}}
\newcommand{\udiag}{\mathcal{U}_{d}}
\newcommand{\uc}{\mathcal{U}_k}
\newcommand{\aaa}{\mathcal{A}}
\newcommand{\bch}{B-C-H}
\newcommand{\bb}{\mathcal{B}}
\newcommand{\kk}{\mathcal{K}}
\newcommand{\J}{\mathcal{J}}
\newcommand{\PP}{\mathcal{P}}

\newcommand{\h}{\mathcal{H}}

\newcommand{\N}{\mathbb N}
\newcommand{\C}{\mathbb C}

\newcommand{\R}{\mathbb R}
\newcommand{\F}{\mathcal{F}}

\newcommand{\I}{\mathcal I}

\newcommand{\longi}{{\rm L}}
\newcommand{\bit}{\begin{itemize}}
\newcommand{\eit}{\end{itemize}}
\newcommand{\be}{\begin{enumerate}}
\newcommand{\ee}{\end{enumerate}}
\newcommand{\bx}[1]{\begin{exer}\rm{#1}}
\newcommand{\ex}{\end{exer}}

\newcommand{\ba}{\begin{array}}
\newcommand{\ea}{\end{array}}
\newcommand{\bc}{\begin{center}}
\newcommand{\ec}{\end{center}}

\newcommand{\s}{\beta}

\newcommand{\D}{\mathcal{D}}

\newcommand{\bq}{\begin{equation}}
\newcommand{\eq}{\end{equation}}

\begin{document}
 
\baselineskip=17pt

\title[Unitary subgroups and compact self-adjoint operators]{Unitary subgroups and orbits of compact self-adjoint operators}

\author[T. Bottazzi]{Tamara Bottazzi}
\address{Instituto Argentino
	de Matem\'atica ``Alberto P. Calder\'on''\\ Saavedra 15 3$^\text{er}$ piso\\
	(C1083ACA) Ciudad Aut\'onoma de Buenos Aires, Argentina
	}
\email{tpbottaz@ungs.edu.ar}

\author[A. Varela]{Alejandro Varela}
\address{Instituto de Ciencias, Universidad Nacional de General Sarmiento\\
	J.M. Gutierrez 1150\\ 
	(B1613GSX) Los Polvorines, Pcia. de Buenos Aires, Argentina
	\\and\\ Instituto Argentino
	de Matem\'atica ``Alberto P. Calder\'on'', Saavedra 15 3$^\text{er}$ piso,
	(C1083ACA) Ciudad Aut\'onoma de Buenos Aires, Argentina} 

\email{avarela@ungs.edu.ar}

\subjclass[2010]{MSC: Primary: 22F30, 22E10, 51F25, 53C22. Secondary: 47B10, 47B15.}
\keywords{
	unitary groups, Lie subgroups, unitary orbits, geodesic curves, minimal operators in quotient spaces.}

\date{}

\begin{abstract} 
Let $\h$ be a separable Hilbert space, and $\D(\bah)$ the anti-Hermitian bounded diagonals in some fixed orthonormal basis and $\kh$ the compact operators.  We study the group of unitary operators
$$
\ukd=\{u\in \uh: 
\exists\ D\in\dah \text{ such that } u-e^D \in\kh\}
$$ 
in order to obtain a concrete description of short curves in unitary Fredholm orbits 
$\ob=\{ e^K b e^{-K}:K\in\kah\}$ of a compact self-adjoint operator $b$ with spectral multiplicity one. We consider the rectifiable distance on $\ob$ defined as the infimum of curve lengths measured with the Finsler metric defined by means of the quotient space $\kah/\mathcal{D}(\kah)$. Then for every $c\in\ob$ and $x\in T(\ob)_c
$ there exist 
 a minimal lifting $Z_0\in \bah$ (in the quotient norm, not necessarily compact) such that $\gamma(t)=e^{t Z_0}\, c\, e^{-t Z_0}$ is a short curve on $\ob$ in a certain interval.
\end{abstract} 

\maketitle 

\section{Introduction}
Let $\bb(\h)$ be the algebra of bounded operators on a separable Hilbert space $\h$, $\kk(\h)$ and $\uh$ the compact and unitary operators respectively. If an orthonormal basis is fixed we can consider matricial representations of each $A\in\bb(\h)$ and diagonal operators which we denote with $\dbh$.

Consider the following subset of the unitary group $\uh$ of $\bb(\h)$:
\begin{equation}\label{def: definicion del subgrupo Ukd}
\ukd=\{u\in\uh: \exists \  D\in\dah \text{ such that } u-e^D\in\kh  \}. 
\end{equation}
In the present work we prove that $\ukd$ is a subgroup of $\uh$. Moreover, $\ukd$ is closed, pathwise connected and shares the topology of $\uh$ given by the operator norm. Therefore $\ukd$ is a Lie subgroup in the sense of \cite{neeb} and \cite{neeb_pianzola}.

We did not find any reference to the subgroup $\ukd$ mentioned in the literature and so we included here a detailed study of it. In Theorem \ref{teo: ukd es subgrupo de Lie} we prove that $\ukd$ is a Lie subgroup of $\uh$ according to the definition mentioned before.
The Lie algebra of $\ukd$ turns to be $\kah + \dah$ which is not complemented in $\bah$ and therefore a stronger notion of Lie subgroup cannot be used (see Proposition \ref{prop: el algebra de Lie de ukd es kah+dah}).

This subgroup admits a generalization to $\mathcal{U}_{\mathcal{J},\mathcal{A}}$ for certain ideals $\mathcal{J}$ and subalgebras $\mathcal{A}$ of $\bh$ (see \ref{rem: generalizacion del subgrupo ukd}).

Our particular interest in $\ukd$ relies on the geometric study of the orbits
$$
\mathcal{O}_b^{\mathcal{V}}=\{u b u^* : u\in \mathcal{V} \}
$$ 
where $b$ a self-adjoint operator. 
If the spectrum of $b$ is finite $\ob$ is a complemented submanifold of $b+\kh$ (see \cite{andruchow_larotonda_The_rectifiable_distance}). If we consider a compact diagonal self-adjoint operator $b$ with spectral multiplicity one then the orbit $\ob$ can have a smooth structure (see Lemma 1 in \cite{bottazzi_varela_DGA}).

The subgroup $\ukd$ has the following properties.
\begin{itemize}
	\item 
		If  $\uc=\left\{u\in\uh: \exists\  D\in\mathcal{D}(\kah) \text{ such that } u-1\in\kh \right\}$, the following orbits coincide
	\begin{equation*}
		\begin{split} 
		\ob &=\mathcal{O}_b^{\uc}=\{u b u^*  : u\in\uc \}\\
			&=	\mathcal{O}_b^{\ukd}=\{u b u^*  :  u\in\ukd\}
	  \end{split}				
	  \end{equation*}

	\item The natural Finsler metric defined in $T(\mathcal{O}_b^{\ukd})_1$ and $T(\mathcal{O}_b^{\uc})_1$ by means of the quotient norm coincides if $b$ is a compact self-adjoint diagonal operator and we consider the identifications of the tangent spaces with the quotients
	\begin{equation*}
	T(\mathcal{O}_b^{\uc})_1\cong(T\ob)_c\cong (T\uc)_1/(T\I_b)_1 = \kah/\D(\kah)
	\end{equation*}
	\begin{equation*}
	\begin{split}
	T(\mathcal{O}_b^{\ukd})_1&\cong(T\ukd)_1/(T\I_b)_1 \cong \left( \kah+\dah\right)/\dah\\ &\cong\kah/\dah
	\end{split}
	\end{equation*}
	(see Remark \ref{rem: ukd y uk misma metrica de Finsler en Ob} for details).

\end{itemize}
%

These properties allow 
the construction of minimum length curves of $\ob$ considering the rectifiable distance defined in the Preliminaries (see \eqref{def: distancia rectificable}). 

%

Next we describe minimal vectors of the tangent space and their relation with the short curves in these homogeneous spaces.
 We say that a self-adjoint operator $Z\in\bb(\h)$ {\sl is minimal for} a subalgebra $\mathcal{A}\subset\bb(\h)$ if
\begin{equation}\label{def: operador minimal autoajunto}
\|Z\|=\inf_{D\in\mathcal{A}}\|Z+D\| ,
\end{equation}
for $\|\cdot \|$ the usual operator norm in $\bb(\h)$.
Given a fixed $Z$ we say that $D_0\in\mathcal{A}$ {\sl is minimal for} $Z$ if $\|Z+D_0\|=\inf_{D\in\mathcal{A}}\|Z+D\|$, that is, if $Z+D_0$ is minimal for $\mathcal{A}$. These minimal operators $Z$ allow the concrete description of short curves $\gamma(t)=e^{itZ}A e^{-itZ}$ in the unitary orbit $\oa$ of a some fixed self-adjoint operator $A\in\bh$, when considered with a certain natural Finsler metric (see \eqref{eq: def metrica de Finsler}, \cite{dmr1}, \cite{ andruchow_larotonda_The_rectifiable_distance}  and \cite{bottazzi_varela_DGA} for details and different examples).

If we fix an orthonormal basis in $\h$ we can consider matricial representations and diagonal operators in $\bb(\h)$. In \cite{bottazzi_varela_DGA} we studied the orbit $\oa$ of a diagonal compact self-adjoint  operator $b\in\bb(\h)$ under the action of the Fredholm unitary subgroup $\uc=\{e^K:K\in \kah\}$ where $\kah$ denotes the compact anti-Hermitian operators. 
We used a particular element $Z_r\in\kah$ with the property that there does not exist a compact diagonal $D_0$ such that the quotient norm 
$\|Z_r+D_0\|=\inf_{D\in\dkh}\|Z_r+D\|$ is attained. This example posted an interesting geometric question, since the existence of such minimal compact diagonal $D_0$ would allow the explicit description of a short path with initial velocity $[Z_r,b]$ (see \cite{dmr1,bottazzi_varela_LAA}).

Using that $\lim \left(Z_r\right)_{jj}$ converges to a non-zero constant when $j\to\infty$ we showed in \cite{bottazzi_varela_DGA} that the curve parametrized by 
$$
\beta(t)=e^{tZ_r}be^{-tZ_r}
$$ 
with $|t|\leq\frac{\pi}{2\|Z_r\|}$, is still a geodesic even though $Z_r$ is not a minimal operator. Moreover, $\beta$ can be approximated uniformly by minimal length curves of finite matrices $\beta_n$ (with minimal initial velocity vectors) satisfying $\beta_n(0)=\beta(0)=b$ and $\beta_n'(0)=\beta'(0)$.

Nevertheless, in the same paper, we showed examples of compact operators $Z_o$ whose unique minimal diagonals had several limits. In these cases the techniques used with $Z_r$ were not enough to prove either that $\gamma(t)=e^{tZ_o}be^{-tZ_o}$ was a short curve nor that $\gamma$ could be approximated by curves of matrices.

In the present work we describe short curves that include those cases. In order to do so we consider the unitary subgroup $\ukd$.
The action of this group on a diagonal  self-adjoint operator $b$ produces the same orbit as $\uc$ but permits a concrete description of geodesics using minimal operators of its Lie algebra $\kah+\dah$ (see \ref{teo: exp(z0)b exp(z0) es corta en Ob} and \ref{cor: curvas minimales con comienzo en c}).

%

\section{Preliminaries}
Let $(\h,\left\langle ,\right\rangle)$ be a separable Hilbert space.
As usual, $\bb(\h)$, $\uu(\h)$ and $\kk(\h)$ denote the sets of bounded, unitary and  compact operators on $\h$. We denote with $\left\|\cdot\right\|$ the usual operator norm in $\bb(\h)$. It should be clear from the context the use of the same notation $\left\|\cdot\right\|$ to refer to the operator norm or the norm on the Hilbert space $\|h\|=\langle h, h\rangle^{1/2}$ for $h\in\h$.

Given $\aaa\subset\bb(\h)$, we use the superscript $^{ah}$  (respectively $^{h}$) to note the subset of anti-Hermitian (respectively Hermitian) elements of $\aaa$.

Consider the Fredholm subgroup of $\uu(\h)$ defined as
$$\uc=\{u\in \mathcal{U}(\h): u-I\in \kk(\h)\}=\{u\in\uh: \exists\ K\in\kah, u=e^K\}$$
(see \cite{andruchow_larotonda_The_rectifiable_distance} and Proposition \ref{prop: uc es ek}).

$\uu(\h)$ is a Lie-Banach group and its Lie algebra $T_1\left(\uu(\h)\right)=\bah$. We consider the usual analytical exponential map $\exp:\bah\to\uu(\h)$, given for any $X\in\bah$ by $\exp(X)=\sum_{n=0}^{\infty}\frac{1}{n!}X^n=e^X$.
Then $\bah$ can be made into a contractive Lie algebra (i.e., $\left\|[X,Y]\right\|_c\leq \left\|X\right\|_c\left\|Y\right\|_c$ for all $X,Y\in \bah$) defining
$\left\|\cdot\right\|_c:=2\left\|\cdot\right\|$. 
Then, by Proposition 1.29 in \cite{beltita_Smooth_homogeneous_structures} 
$$\left\|\log\left(e^Xe^Y\right)\right\|_c\leq -\log\left(2-e^{\left\|X\right\|_c+\left\|Y\right\|_c}\right)$$
if $\left\|X\right\|_c+\left\|Y\right\|_c< {\log 2}$. Consequently, it can be proved that if
\begin{equation}\label{cerca de 0}
\left\|X\right\|+\left\|Y\right\|< \frac{\log 2}2
\end{equation}
the Baker-Campbell-Hausdorff (\bch) series expansion converges absolutely for all $X,Y\in \bah$. This \bch\ series can be defined as
$$\log(e^Xe^Y)=\sum_{n=1}^{\infty}c_n(T),$$
where each $c_n$ is a polynomial map of $\bah\times\bah$ into $\bah$ of degree $n$, $\forall\ n\in \N$. For instance, the first terms are:
$$\left\{\begin{array}{l}
c_1(X,Y)=X+Y,\\ 
c_2(X,Y)=\frac{1}{2}[X,Y],\\
c_3(X,Y)=\frac{1}{12}[X,[X,Y]]+\frac{1}{12}[Y,[Y,X]].
\end{array}\right.$$
Also, each $c_n$ is a sum of commutators for all $n>1$. Therefore, the formula of the series can be rewritten as follows
\begin{equation}\label{bch}
\log(e^Xe^Y)=X+Y+\sum_{n=2}^{\infty}c_n(T).
\end{equation}
To see the complete general expression or other properties of the \bch\ series for Lie algebras see \cite{beltita_Smooth_homogeneous_structures} or \cite{varadarajan_Lie_groups}.

\begin{defi}\label{def: suficientemente cerca de cero}
	Given $X\in \bah$ we will say that $X$ is sufficiently close to $0$ if $\left\|X\right\|<\frac{\log 2}4$.
\end{defi}
Using the previous definition the \bch\ series \eqref{bch} converges for every $X,Y\in \bah$ sufficiently close to $0$, since this condition implies \eqref{cerca de 0}.

We define the unitary Fredholm orbit of a fixed self-adjoint $A\in \bb(\h)$ as
\begin{equation}
\mathcal{O}_A=\{uAu^*:u\in\mathcal{U}_k(\h)\}\ \subset\ A+\kk(\h).\label{orbita}
\end{equation}
Considering the action $\pi_b:\uu_k\to\oa$, $\pi_b(u)=L_{u}\cdot b = ubu^*$ then $\oa$ becomes a homogeneous space in some cases. 
If $A$ has finite spectrum then $\oa$ is a submanifold of $A+\kh$ (see Theorem 4.4 in \cite{andruchow_larotonda_The_rectifiable_distance}) and if $A$ is a compact operator with spectral multiplicity one then $\oa$ has a smooth structure (see Lemma 1 in \cite{bottazzi_varela_DGA}).

Denote by $[\cdot ,\cdot ]$ the commutator operator in $\bb(\h)$, that is, for any $T,S\in \bb(\h)$, $[T,S]=TS-ST.$

For each $b\in \oa$, the isotropy group $\I_b$ is
\begin{equation*}
\begin{split}
\I_b&=\{u\in \uc:\ ubu^*=b\}=\{e^K\in \uc:\ K\in \kah,\ [K,b]=0 \}
\\
&=\{b\}'\cap \kah,
\end{split}
\end{equation*}
where $\{b\}'$ is the set of all operators in $\bb(\h)$ that commute with $b$ (i.e., $[T,b]=0$).

For each $b\in \oa$, its tangent space is
$(T\oa)_b=\{Yb-bY: Y\in\kah \}$
and can be identified as follows
$$
(T\oa)_b\cong(T\uc)_1/(T\I_b)_1\cong\kah/\left(\{b\}'\cap\kh^{ah}\right).
$$
In this context we consider the following Finsler metric defined for $x\in (T\oa)_b$ as
\begin{equation}\label{eq: def metrica de Finsler}
\begin{split}
\|x\|_b&=\inf\left\{\|Y\|: Y\in\kah \hbox{ such that } [Y,b]=x\right\}\\
&=\inf_{
	C\in \left(\{b\}'\cap\kh^{ah}\right)}
\ 
\left\|Y_0+C\right\|.
\end{split}
\end{equation}

where $Y_0+C$ is any element of the class $[Y_0]= \{Y\in\kah:  [Y,b]=x\}$.
Note that this norm is invariant under the action of $\uc$. 

An element $Z\in \bb(\h)^{ah}$ such that $[Z,b]=x$ and $\left\|Z\right\|=\left\|x\right\|_b$ is called a minimal lifting for $x$. This operator $Z$ may not be compact and/or unique (see \cite{bottazzi_varela_LAA}).
Consider piecewise smooth curves $\beta:[a,b]\to \oa$. We define 
\begin{equation}\label{def: longitud rectificable}
\longi(\beta)=\int_a^b\left\|\beta'(t)\right\|_{\beta(t)}\ dt\ 
\text{ , and }
\end{equation}
\begin{equation}\label{def: distancia rectificable}
\dist(c_1,c_2)=\inf\left\{\longi(\beta):\beta\ {\rm is\ smooth},\beta(a)=c_1,\beta(b)=c_2\right\} 
\end{equation}
as the rectifiable length of $\beta$ and distance between two points $c_1, c_2\in\oa$, respectively.
In this context, we will call short curves those piecewise smooth curves $\gamma:[r,s]\to \mathcal{O}_A$ such that  $\longi(\gamma|_{[t_1,t_2]}) \leq \dist (\gamma(t_1), \gamma(t_2))$ for every subinterval $[t_1,t_2]\subset [r,s]$.

If $\aaa$ is any $C^*$-subalgebra of $\bb(\h)$ and $\left\{e_k\right\}_{k=1}^{\infty}$ is a fixed orthonormal basis of $\h$, we denote with $\D(\aaa)$ the set of diagonal operators with respect to this basis, that is
$$\D(\aaa)=\left\{T\in \aaa:\ \left\langle Te_i,e_j\right\rangle=0\ ,\ \text{ for all } i\neq j\right\}.$$
Given an operator $Z\in \aaa$, if there exists an operator $D_1\in \D(\aaa)$ such that
$$\left\|Z+D_1\right\|\leq \left\|Z+D\right\|$$
for all $D\in \D(\aaa)$, we say that $D_1$ is a best approximant of $Z$ in $\D(\aaa)$. The operator $Z+D_1$ satisfies 
$$\left\|Z+D_1\right\|={\rm dist}\left(Z,\D\left(\aaa\right)\right),$$
and $Z+D_1$ is a minimal operator in the class $[Z]$ of the quotient space $\aaa/\D(\aaa)$, or similarly we say that $D_1$ is minimal for $Z$. 

These minimal operators play an important role in the concrete description of minimal length curves on $\oa$ (see \cite{dmr1} and \cite{andruchow_larotonda_The_rectifiable_distance}).

If $Z$ is anti-Hermitian it holds that 
$${\rm dist}\left(Z,\D\left(\aaa\right)\right)={\rm dist}\left(Z,\D\left(\aaa^{ah}\right)\right),
$$
since $\left\|Im(X)\right\|\leq \left\|X\right\|$ for every $X\in \aaa$.

Let $T\in \bb(\h)$ and consider for the fixed basis of $\h$ the coefficients $T_{ij}=\left\langle Te_i,e_j\right\rangle$ for each $i,j\in \N$. This defines an infinite matrix $\left(T_{ij}\right)_{i,j\in \N}$ such that their $j$th-column and $i$th-row of $T$ are the vectors in $\ell^2$ given by $c_j(T)=\left(T_{1j},T_{2j},...\right)$ and $f_j(T)=\left(T_{i1},T_{i2},...\right)$, respectively.

We use $\sigma(T)$ and $R(T)$ to denote the spectrum and range of $T\in \bh$, respectively. 

We define $\Phi:\bb(\h)\to \D(\bb(\h))$, $\Phi(X)= {\diag}(X)$, as the map that builds a diagonal operator with the same diagonal as $X$ (i.e., $\Phi(X)_{ii}=\diag(X)_{ii}=X_{ii}$ and $0$ elsewhere). For a given bounded sequence $\{d_n\}_{n\in\N}\subset \C$ we denote with ${\diag}\big( \{d_n\}_{n\in\N}\big)$ the diagonal (infinite) matrix with $\{d_n\}_{n\in\N}$ in its diagonal  and $0$ elsewhere.

\section{The unitary subgroup {$\ukd$}  }
Recall the unitary Fredholm group 
$$
\uc=\{u\in\uu(\h): u-1\in\kk(\h)\}
$$

(see \cite{andruchow_larotonda_The_rectifiable_distance} and \cite{bottazzi_varela_DGA}) and define the following subsets of $\uu(\h)$:
\begin{equation}\label{def: subconjuntos de operadores unitarios}
\begin{split}
\ukd &=\{u\in \uu(\h):\exists\ D\in \D(\bah)\ {\text{ such\ that }}\ u-e^D\in \kk(\h)\}, \\
\udiag &=\{u\in \uu(\h):\exists\ D\in \D(\bah)\ {\text{ such\ that }}\ u=e^D\}\\
&=\uu(\h)\cap \D(\bb(\h))\\
\uda &=\left\{u\in \uu(\h):\exists\ K\in \kah\ {\text{ and }}\ D\in \D(\bah)\ \right.
\\ &\left.\hskip6cm {\text{ such\ that }} u=e^{K+D}\right\}.
\end{split}
\end{equation}

Also denote with
$$\oo^{\F}_b=\{ubu^*:\ u\in \F\},$$
where $\F$ is any of the sets of unitary operators defined in \eqref{def: subconjuntos de operadores unitarios}. The main purpose of this section is the study of these unitary sets and its relations.

The following Proposition has been proved in \cite{bottazzi_varela_DGA} using arguments of Lemma 2.1 in \cite{andruchow_larotonda_The_rectifiable_distance}.
\begin{prop} \label{prop: uc es ek}
$\uc=\{e^K: K\in \kah, \left\|K\right\|\leq \pi\}$.
\end{prop}
\begin{rem} \label{obs2}
Let $S_0\in \kah$ and $D_0\in \dah$. Then, the exponential series 

$\sum_{n=0}^{\infty}\frac{1}{n!}(S_0+D_0)^{n}$ converges absolutely and
\begin{equation*}\label{eq: partiendo la serie exponencial S_0 mas D_0}
\begin{split}
\sum_{n=0}^{\infty}&\frac{1}{n!}(S_0+D_0)^{n}=\\
&=\sum_{n=0}^{\infty}\frac{1}{n!}\left(S_0^{n}+\binom{n}{1} S_0^{n-1} D_0+\dots+\binom{n}{n-1} S_0 D_0^{n-1}+D_0^n\right)
\\
&=\sum_{n=0}^{\infty}\frac{1}{n!}S_0\left(S_0^{n-1}+\binom{n}{1} S_0^{n-2} D_0+\dots+\binom{n}{n-1} D_0^{n-1}\right)+\frac1{n!}D_0^n
\\
&=S_0\underbrace{\sum_{n=1}^{\infty}\frac{1}{n!}\left(S_0^{n-1}+\binom{n}{1} S_0^{n-2} D_0+\dots+\binom{n}{n-1}  D_0^{n-1}\right)}_{\begin{array}{c}
	\scriptscriptstyle|\scriptscriptstyle|\\ \Psi(S_0,D_0)\end{array}}
+\sum_{n=0}^\infty\frac1{n!}D_0^n\\
&=S_0 \Psi(S_0,D_0)+e^{D_0}.
\end{split}
\end{equation*}
with $S_0 \Psi(S_0,D_0)\in \kh$.
\end{rem}

\begin{prop}\label{prop: Ukd es un grupo y es Uk.Ud}
$\ukd$ is a unitary subgroup of $\uh$ and it equals 

\begin{equation*}
\begin{split}
\uc \udiag =\left\{ u\in \uu(\h)\right. : &\exists\ K\in \kah, \left\|K\right\|\leq \pi,\ \text{ and }
\\
\hskip6cm & \left. D\in \D(\bah)\
\text{ such that }\ u=e^{K}e^D\right\}.
\end{split}
\end{equation*}

Moreover $$\ukd=\uc \udiag=\udiag \uc.$$
\end{prop}

\begin{proof}
Let $u\in\ukd$, then there exists $D\in \D(\bah)$ such that $u-e^D\in \kk(\h)$. Then
$$ue^{-D}-1\in \kk(\h)\Rightarrow \exists K\in\kah, \left\|K\right\|\leq \pi \text{ such that } \ ue^{-D}=e^K$$
$$\Rightarrow u=e^Ke^{D},$$ 
and therefore $u\in \uc \udiag$.

Conversely, if there exists $K'\in\kah$, 
and $D'\in \D(\bah)$ such that $u=e^{K'}e^{D'}\in \uu(\h)$, then
$$ 
ue^{-D'}=e^{K'}\in \uc\Rightarrow ue^{-D'}-1\in \kk(\h)\Rightarrow (ue^{-D'}-1)e^{D}=
u-e^{D'}\in \kk(\h).$$
These calculations prove that $\ukd=\uc.\udiag$. 

Similar computations (with left multiplication of $e^{-D}$) lead to the equality $\ukd=\udiag\uc$.

Now we will prove that $\ukd$ is a group. Let $u,v\in \ukd$:
\bit
\item There exists $D\in\D(\bah)$ such that $u-e^D\in \kk(\h)$. Then $u^*-e^{-D}\in \kk(\h)\Rightarrow u^*\in \ukd$.
\item Using that $\ukd=\uc \udiag$ we can write $u=e^{K_1}e^{D_1}$ and $v=e^{K_2}e^{D_2}$, with $K_1,K_2\in \kah$ and $D_1,D_2\in \D(\bah)$. Then using Remark \ref{obs2}
$$uv=e^{D_1+D_2}+K_1\Psi(K_1,D_1)+e^{D_2}K_2\Phi(K_2,D_2)+K_1\Psi(K_1,D_1)K_2\Phi(K_2,D_2).$$
Therefore $uv-e^{D_1+D_2}\in \kk(\h)$ which implies that 
$uv\in \ukd.$
\eit
Then, $\ukd$ is a unitary subgroup of $\uh$.
\end{proof}

\begin{prop} \label{propiedades}
Let $\ukd$, $\udiag$ and $\uda$ be as defined in 
(\ref{def: subconjuntos de operadores unitarios}), then the following statements hold:
\be 
\item $\udiag\subsetneq \ukd$.
\item $\uc\subsetneq\ukd$.
\item $\uda\subseteq\ukd$.
\item If $u\in \ukd$ then $u=K'+D'$, with $K'\in \kk(\h)$ and $D\in \D(\uu(\h))$.
\item For every $K\in\kah$ and $D\in \D(\bah)$ there exists $K'\in \kah$ such that $e^Ke^D=e^De^{K'}$.
\item $\uc\subsetneq \uda$ 
\item $\uc=\{u\in \uu(\h):\exists D\in \D(\kah)\ {\rm such\ that}\ u-e^D\in \kk(\h)\}$.

\ee
\end{prop}

\begin{proof}
\be
\item It is apparent.
\item $u\in \uc\Leftrightarrow u-1\in\kk(\h)\Leftrightarrow u-e^0\in\kk(\h)$.
\item Let $e^{K+D}$ with $K\in  \kah$ and $D\in\D(\bah)$. Then
$$
e^{K+D}=1+(K+D)+\frac{1}{2!}(K+D)^2+...=e^D+K\Psi(K,D)\Rightarrow
$$
$$\Rightarrow e^{K+D}-e^D\in\kk(\h)\Rightarrow e^{K+D}\in\ukd.$$
\item If $u\in \ukd\Rightarrow u-e^D\in \kk(\h)\ {\rm with}\ D\in \D(\bah)\Rightarrow \exists K'\in \kk(\h)/\ u=K'+e^D$. 
\item If $K\in  \kah$ and $D\in\D(\bah)$ then Proposition \ref{prop: uc es ek} implies that $e^K-1\in \kk(\h)$ which gives $(e^K -1)e^D=e^K e^D-e^D\in \kk(\h)$ and $e^{-D}e^Ke^D-1\in \kk(\h)$. Then, $e^{-D}e^Ke^D\in \uc$ and there exists $K'\in\kah$ such that $e^{-D}e^Ke^D=e^{K'}$. The result follows easily. 
\item It is apparent. 
\item If $u\in\uc$ then $u-1=u-e^0\in\kk(\h)$ with $0\in\D(\kah)$ and then $u\in \{u\in \uu(\h):\exists D\in \D(\kah)\ {\rm such\ that}\ u-e^D\in \kk(\h)\}$. Conversely, let $u\in \uu(\h)$ and $D\in \D(\kah)$ such that $u-e^D\in \kk(\h)$. Then 
$$u-1=u-e^D+e^D-1\in \kk(\h),$$
since $e^D\in \uc$, which completes the proof.
\ee
\end{proof}

\begin{prop}\label{prop: diagonales compactas d para e^-d e^d entre e^K y e^D}
 Let $K_1, K_2\in\kah$, $D_1, D_2\in \D(\bah)$, then the following statements are equivalent:
\begin{enumerate}
  \item [a) ] $e^{K_1} e^{D_1}=e^{K_2} e^{D_2}$  
\item [b) ] There exists $d\in \D(\kah)$ such that $$e^{K_2}=e^{K_1}e^{-d}\ \ \text{ and } \ \ e^{D_2}=e^d e^{D_1}=e^{d+D_1}.$$
\end{enumerate}
 
\end{prop}
\begin{proof}
b) $\implies$ a) is apparent after computing $e^{K_2}e^{D_2}$.

Let us consider a) $\implies$ b).

 If $e^{K_1}e^{D_1}=e^{K_2}e^{D_2}$ then $e^{D_1-D_2}=e^{-K_1}e^{K_2}$. Since $e^{-K_1}, e^{K_2}\in\uc$ which is a group, then there exists $K_{1,2}\in\kah$ such that $\|K_{1,2}\|\leq \pi$ and $e^{-K_1}e^{K_2} = e^{K_{1,2}}$ (see Proposition \ref{prop: uc es ek}). Moreover, there exists a diagonal $D_{1,2}\in\D(\bah)$ with $\|D_{1,2}\|\leq \pi$ such that $e^{D_1-D_2}=e^{D_{1,2}}$. Therefore
$$
e^{K_{1,2}} = e^{-K_1}e^{K_2} = e^{D_1-D_2}=e^{D_{1,2}}
$$
with $K_{1,2} \in \kah $ and $D_{1,2}\in \D(\bah)$. Using Theorem 3.1 in \cite{chiumiento_normal_logarithms} we can conclude that $|K_{1,2}|=|D_{1,2}|$ which implies that $K_{1,2}$ and $D_{1,2}$ are both diagonal and compact operators. If we chose $-d=D_{1,2}\in \D(\kah)$, then
$$
e^{K_2}=e^{K_1}e^{D_1-D_2}=e^{K_1}e^{D_{1,2}}=e^{K_1}e^{-d}
$$
and 
$$
e^{D_2}=e^{D_2-D_1} e^{D_1}=e^{-D_{1,2}}e^{D_1}=e^d e^{D_1}
$$
which proves the proposition.
\end{proof}

\begin{prop}\label{prop: u en ukd si compacto fuera diag y mod de diag tiende a 1} 
	Let $u\in\uh$. Then the following statements are equivalent
	\begin{enumerate}
		\item[ a)] $u\in\ukd $
		\item[ b)] $u-\diag(u) \in \kh$ and $|u_{j,j}|\underset{j\to\infty}{\to} 1$.
	\end{enumerate}
	
\end{prop}
\begin{proof}
	
	a) $\Rightarrow$ b) If $u\in\ukd$ then there exists $D\in\dah$ such that $u-e^D\in\kh$. Then $\diag(u-e^D)_{j,j}=u_{j,j}-e^D_{j,j}\underset{j\to\infty}{\to} 0$ and therefore $\diag(u)-e^D\in\kh$ and $|u_{j,j}|\underset{j\to\infty}{\to} 1$. Then since 
	$$
	u-e^D=u-\diag(u)+ \underbrace{\diag(u)-e^D}_{\in\kh} \in\kh ,
	$$ 
	we obtain that $u-\diag(u)\in\kh$.
	
	\medskip
	b) $\Leftarrow$ a) If $u_{j,j}\underset{j\to\infty}{\to} 1$ then there exists $D\in\dah$ such that $\left(u_{j,j}-e^D_{j,j}\right)\underset{j\to\infty}{\to} 0$ (for example take $e^D_{j,j}=\frac{1}{|u_{j,j}|} u_{j,j}$ when $j$ is sufficiently large). Then $\diag(u)-e^D\in\kh$. Using the hypothesis that $\diag(u)-e^D\in\kh$, then 
	$$
	\underbrace{u-\diag(u)}_{\in\kh}+\underbrace{\diag(u)-e^D}_{\in\kh}=u-e^D\in\kh
	$$
	and therefore $u\in\ukd$.
\end{proof}

\begin{rem}
The previous proposition allow us to prove easily that $\ukd\subsetneq \uh$ since the block diagonal defined symmetry $u=\left(\begin{smallmatrix}
s & 0 & 0 &\dots \\ 
0 & s & 0 &\dots  \\
0 & 0 & s&  \vspace*{-.2cm} \\
\vdots & \vdots &  &\ddots \\
\end{smallmatrix}\right)$ with $s=\left(\begin{smallmatrix}
0 & 1   \\ 
1 & 0 
\end{smallmatrix}\right)$ clearly does not satisfy conditions b) of Proposition \ref{prop: u en ukd si compacto fuera diag y mod de diag tiende a 1}, but $u\in\uh$.
\end{rem}

%
%
%

The following proposition is a consequence of 
results present in \cite{chiumiento_normal_logarithms}.
\begin{prop}\label{prop: consecuencias de e^d e^k=e^k'e^d}
 Let $K$, $K'\in\kah$ satisfying  $\|K\|, \|K'\|\leq \pi$ and $D\in\dah$ be such that $e^D e^K=e^{K'}e^{D}$. Then $\|K\|=\|K'\|$ and
\begin{itemize}
 \item[a)] if $\|K\|=\|K'\|= \pi$, then
\begin{enumerate}
\item $|K|=e^{-D}|K'|e^D$,
\item $v\in\h$ is an eigenvector of $K$ with corresponding eigenvalue $\lambda\in i\R$, $|\lambda|<\pi$  \ $\Longleftrightarrow$ \ $e^D v$ is an eigenvector of $K'$ with corresponding eigenvalue $\lambda\in i\R$, $|\lambda|<\pi$,
\item if $E_X$ is the spectral measure of the operator $X$, then
$$
K-e^{-D}{K'}e^D=2\pi i \left(E_K(\R+i\pi)-E_{e^{-D}{K'}e^D}(\R+i\pi)\right)
$$
\end{enumerate}
 \item[b) ] and moreover, if $\|K\|=\|K'\|< \pi$ then $K=e^{-D}K'e^{D}$. 
\end{itemize}

\end{prop}
  
\begin{proof}
Observe first that since $e^D e^K=e^{K'}e^{D}$ then
\begin{equation}\label{eq: exp k = exp(exp -d k' exp d)}
e^K=e^{-D}e^{K'}e^{D}=e^{e^{-D}K'e^{D}} 
\end{equation}
and therefore $|K|=|e^{-D}K'e^{D}|=e^{-D}\ |K'|\ e^{D}$ (see Theorem 3.1 i) in \cite{chiumiento_normal_logarithms}) which implies $\|K\|=\|K'\|$.
\begin{itemize}
 \item[a) ]
\begin{enumerate}
\item This is a direct consequence of \eqref{eq: exp k = exp(exp -d k' exp d)}, the fact that $\sigma(K)$ and $\sigma(e^{-D} K' e^D)$ are contained in $\mathcal{S}=\{z\in\C: -\pi\leq \text{Im}(z)\leq \pi\}$ and Theorem 3.1 i) of \cite{chiumiento_normal_logarithms}.
 \item Consider $\lambda\in\sigma(K)\subset i\R$, $|\lambda|<\pi$ and $v\in\h$ such that $Kv=\lambda v$. Then $e^K v=e^\lambda v$ and the equation \eqref{eq: exp k = exp(exp -d k' exp d)} imply that $e^\lambda$ is an eigenvalue of $e^{e^{-D}K'e^{D}}$ with eigenvector $v$. Then $\lambda\in \sigma(e^{-D}K'e^{D})$ (because $|\lambda|<1$) with eigenvector $v$. Therefore  $\lambda\in \sigma(K')$ with eigenvector $e^Dv$. The other implication follows similarly.
\item This statement follows from $\sigma(K), \sigma(e^{-D} K' e^D)\subset \mathcal{S}$, Remark 2.4 and Theorem 4.1 in \cite{chiumiento_normal_logarithms}.
\end{enumerate}

\item[b) ]
If the strict inequality $\|K\|=\|K'\|< \pi$ holds then \eqref{eq: exp k = exp(exp -d k' exp d)} and Corollary 4.2 iii) in \cite{chiumiento_normal_logarithms} imply directly that $K=e^{-D}K'e^{D}$.

\end{itemize}
\end{proof}

\begin{cor}
 Let $K$, $K'\in\kah$, $\|K\|, \|K'\|\leq \pi$ and $D\in\dah$. Then
 \begin{itemize}
 \item[a)] if $\|K\|=\|K'\|= \pi$, the following equivalence holds 
$$e^D e^K=e^{K'}e^{D}\ \Longleftrightarrow\ 
K-e^{-D}{K'}e^D=2\pi i \left(E_K(\R+i\pi)-E_{e^{-D}{K'}e^D}(\R+i\pi)\right)
$$
\item[b) ] and if $\|K\|, \|K'\|<\pi$,  the following equivalence holds
$$e^D e^K=e^{K'}e^{D}\ \Longleftrightarrow\ 
K=e^{-D}{K'}e^D$$
\end{itemize}
\end{cor}
\begin{proof}
\begin{itemize}
 \item[a) ] If $e^D e^K=e^{K'}e^{D}$ then 
 $$
K-e^{-D}{K'}e^D=2\pi i \left(E_K(\R+i\pi)-E_{e^{-D}{K'}e^D}(\R+i\pi)\right)
$$ 
follows from a) (3) of the previous Proposition \ref{prop: consecuencias de e^d e^k=e^k'e^d}. 

The converse is proved using that $K-2\pi i  E_K(\R+i\pi)=e^{-D}{K'}e^D-E_{e^{-D}{K'}e^D}(\R+i\pi)$ implies that 
$$
e^{K-2\pi i  E_K(\R+i\pi)}=e^{e^{-D}{K'}e^D-2\pi i E_{e^{-D}{K'}e^D}(\R+i\pi)}
$$
and since $K$ commutes with $E_K$ and $e^{-D}{K'}e^D$ with $E_{e^{-D}{K'}e^D}$, follows that
$$
e^{K}e^{-2\pi i  E_K(\R+i\pi)}=e^{e^{-D}{K'}e^D}e^{-2\pi i E_{e^{-D}{K'}e^D}(\R+i\pi)}.
$$
Since $e^{-2\pi i  E_K(\R+i\pi)}=e^{-2\pi i E_{e^{-D}{K'}e^D}(\R+i\pi)}=1$  then
$$
e^{K}=e^{e^{-D}{K'}e^D}=e^{-D}e^{K'}e^D
$$
which ends the proof.
\item[b) ] It is apparent using the previous Proposition \ref{prop: consecuencias de e^d e^k=e^k'e^d} (b) and the fact that $e^{e^{-D}{K'}e^D}=e^{-D}e^{K'}e^D$.
\end{itemize}

\end{proof}


In \cite{thompson_Proof_of_a_conjectured_exponential_formula} Thompson proved for any $X,Y\in M_n(\C)^{ah}$ that there exist unitaries $U,V$ such that
\begin{equation}\label{thompson_Proof_of_a_conjectured_exponential_formula}
e^{X}e^{Y}=e^{U^*XU+V^*YV}.
\end{equation}
Subsequently, in \cite{antezana_larotonda_varela_Thompson} Antezana et. al. proved a generalization of \eqref{thompson_Proof_of_a_conjectured_exponential_formula} for compact operators: 
given $K_1, K_2\in \kah$, there exist unitaries $U_n,V_n$, for $n\in \N$, such that
$$e^{K_1}e^{K_2}=\lim\limits_{n\to\infty}e^{U_n^*K_1U_n+V_n^*K_2V_n}$$
where the convergence is considered in the usual operator norm.
The following proposition adds a new simple case where this equality holds.
\begin{prop}
Let $K\in\kah$ and $D\in \D(\bah)$ and suppose that there exists $\lambda\in i\R$ such that $\lim\limits_{n\to\infty}D_{nn}=\lambda$. Then, there exist unitaries $U_n,V_n$, for $n\in \N$, such that
$$e^{K}e^{D}=\lim\limits_{n\to\infty}e^{U_n^*KU_n+V_n^*DV_n}.$$
\end{prop}
\begin{proof}
Observe that $D-\lambda I\in D(\kah)$. Then, using Theorem 3.1 and Remark 3.3 in \cite{antezana_larotonda_varela_Thompson} there exist unitaries $U_n,V_n$, for $n\in \N$, such that
\begin{align*} 
e^Ke^De^{-\lambda I}&= e^Ke^{D-\lambda I}=\lim\limits_{n\to\infty}e^{U_n^*KU_n+V_n^*(D-\lambda I)V_n}\\ 
 &=\lim\limits_{n\to\infty}e^{U_n^*KU_n+V_n^*DV_n}e^{-\lambda I}.
\end{align*}
Therefore, $e^Ke^D=\lim\limits_{n\to\infty}e^{U_n^*KU_n+V_n^*DV_n}$.
\end{proof}

\begin{prop}\label{uda es grupo}
Let $K_1,K_2\in \kah$ and $D_1,D_2\in \D(\bah)$ such that $K_1+D_1$ and $K_2+D_2$ are sufficiently close to $0$ (see Definition
\ref{def: suficientemente cerca de cero}). Then, there exists $K\in \kah$ such that
$$e^{K_1+D_1}e^{K_2+D_2}=e^{K+D_1+D_2}.$$
\end{prop}

\begin{proof}
Let $K_1,K_2\in \kah$ and $D_1,D_2\in \D(\bah)$ such that $K_1+D_1$ and $K_2+D_2$ are sufficiently close to $0$. Using the \bch\  formula \eqref{bch}, then
$$X=log\left(e^{K_1+D_1}e^{K_2+D_2}\right)=K_1+D_1+K_2+D_2+\sum_{n\geq 2}c_n(K_1+D_1,K_2+D_2).$$

Also, observe that $c_n(K_1+D_1,K_2+D_2)\in \kah$ for every $n$, since $\kk(\h)$ is a two-sided closed ideal and $[D_1,D_2]=0$. Therefore, $X=K+D$, with $K\in \kah$ and $D\in \D(\bah)$ and 
\begin{equation} \label{eq0}
e^{K_1+D_1}e^{K_2+D_2}=e^{K+D}\in \uda.
\end{equation}
In particular $D=D_1+D_2$, since each $c_n$ is a sum of commutators and $\diag\left([A,B]\right)=0$ for every $A,B\in \bah$.
\end{proof}
%

\begin{cor} 
Let $K_1, K_2\in \kah$ and $D_1, D_2\in \D(\bah)$.
\be
\item If $K_1+D_1$ and $K_2+D_2$ are sufficiently close to $0$ (see Definition
\ref{def: suficientemente cerca de cero}), then 
\begin{equation*}
\begin{split}
e^{K_1+D_1}e^{K_2+D_2}&=e^{\tilde{K}+D_1+D_2},\ \text{with}\\ \tilde{K}&=K_1+K_2+\sum_{n\geq 2}c_n(K_1+D_1,K_2+D_2)\in \kah\\   \text{ and }
	 \diag(\tilde{K})&=\diag(K_1+K_2).
	\end{split}
	\end{equation*}

\item If $K_1$ and $D_1$ are sufficiently close to $0$, there exist $K',K''\in\kah$ such that
\begin{equation}\label{eq: e^K_1 e^D_1 = e^D_1 e^K'=e^K''+D_1}
e^{K_1}e^{D_1}=e^{D_1}e^{K'}=e^{K''+D_1}.
\end{equation}
\ee
\end{cor}
\begin{proof}
 These equalities are due to item $(3)$ of the Proposition
\ref{propiedades}, Proposition \ref{uda es grupo} and some calculations from its proof.
\end{proof}

\begin{teo}\label{teo: ukd es cerrado}
 $\ukd$ is pathwise connected and closed in ${\uh}$.
\end{teo}
\begin{proof}
Every $u=e^K e^D\in\ukd$ (with $K\in\kah$ and $D\in\dah$) is connected to $1$ by the curve $\gamma(t)=e^{tK} e^{tD}$, for $t\in[0,1]$.

Consider now the closedness of $\ukd$. Let $\{u_n\}_{n\in\N}\subset\ukd$, $u_n=e^{K_n} e^{D_n}$ for $n\in\N$, $K_n\in\kah$ and $D_n\in \dah$ be a sequence such that $\displaystyle \lim_{n\to \infty}u_n= u_0$ in the usual operator norm in $B(H)$. We will prove that $u_0\in\ukd$.


Since $u_0-u_n=u_0-e^{D_n}+e^{D_n}-u_n$ tends to $0$ as $n\to\infty$ and $e^{D_n}-u_n\in\kh$, for all $n\in\N$,  then dist$\left(\{  u_0-e^{D_n}  \}_{n\in\N}\ , \ \kh \right)=0$.

Observe that 
\begin{equation}\label{eq: distancia de u_0-e^D_n a K(H)}
 \begin{split}
&\dist\left( \left\{\diag(u_0)-e^{D_n}\right\}_{n\in\N}\  , \kh\right)=\\
&
=\dist\left( \left\{\diag(u_0)-\diag(u_n)+\diag(u_n)-e^{D_n}\right\}_{n\in\N}\ , \kh\right)\\
&\leq \inf_{K\in\kh} \| \diag(u_0)-\diag(u_n)\| + \|\diag(u_n)-e^{D_n}-K\|
\end{split}
\end{equation}
for any $n\in\N$.

Note that $u_n-e^{D_n}\in\kh$ which implies that $\diag\left(u_n-e^{D_n}\right)=\diag\left(u_n\right)-e^{D_n}\in\mathcal{D}\left(\kh\right)$. 
Since $u_n\to u_0$ then $\diag(u_n)\to\diag(u_0)$.
Then the first summand in the last inequality (\ref{eq: distancia de u_0-e^D_n a K(H)}) can be chosen to be arbitrarily small for big $n$ and the infimum of the second term is zero because $\diag(u_n)-e^{D_n}\in\kh$. Then
\begin{equation}
\begin{split} 
\label{eq: dist diag u0 - expDn a compactos es cero}
\dist&\left( \left\{\diag(u_0)-e^{D_n}\right\}_{n\in\N}\ , \kh\right)=0
\\
&=\dist\left(\left\{\diag(u_0)-e^{D_n}\right\}_{n\in\N}\  , \ \mathcal{D}\left(\kh\right)\right). 
\end{split}
\end{equation}

%

Moreover, using that $u_n-e^{D_n}\in\kh$, for every $K\in\kh$ holds that
\begin{equation}
 \begin{split}
\dist&\left(u_0-\diag(u_0), \kh\right) =
\inf_{K\in\kh}\|u_0-\diag(u_0)-K\|\\
&\leq \|u_0-\diag(u_0)-u_n+e^{D_n}-K\| \\
&\leq \|u_0-u_n\| + \|e^{D_n}-\diag(u_0)-K\|  
 \end{split}
\end{equation}
Here both summands of the last term can be chosen to be arbitrarily small. It is enough to take $n$ appropriately, since $u_n\to u_0$ and the distance from 
$ \left\{\diag(u_0)-e^{D_n}\right\}_{n\in\N}$ to $\kh$ is null as seen above in \eqref{eq: dist diag u0 - expDn a compactos es cero}.
Then $\dist\left(u_0-\diag(u_0)\ , \ \kh\right)=0$ and therefore 
\begin{equation}\label{eq: u0 - diag u0 es compacto}
u_0-\diag(u_0)\in\kh.                       \end{equation}
If there exists $\delta>0$ such that for a subsequence $\{e^{D_{n_k}}\}_{k\in \N}$, for $k\in\N$, holds that $|(\diag(u_0)-e^{D_{n_k}})_{j,j}|\geq\delta$ for infinite $j\in\N$ which contradicts \eqref{eq: dist diag u0 - expDn a compactos es cero}. Therefore, given $\delta>0$, only finite $n\in\N$ satisfy that $|(\diag(u_0)-e^{D_{n}})_{j,j}|\geq\delta$ for infinite $j\in\N$. Then, if $k\in\N$ and we choose $\delta=\frac1k$, there exists $n_k\in\N$ such that if $n\geq n_k$ then $|(\diag(u_0)-e^{D_{n}})_{j,j}|\geq\frac{1}{k}$ only for finite $j\in\N$. Observe that the subsequence $n_k$ could be chosen to be strictly increasing. For each $k\in\N$, we will define a sub-index $j_k\in\N$ such that $|(\diag(u_0)-e^{D_{n_k}})_{j,j}|<\frac{1}{k}$ for all $j\geq j_k$, $j\in\N$. Moreover $j_k$ can be chosen to be strictly increasing in $k$ and $j_1>1$. Therefore, for each $k\in\N$, there exists $n_k, j_k\in\N$ such that
\begin{equation}\label{eq: cosas que cumplen nk y jk}
|(\diag(u_0)-e^{D_{n_k}})_{j,j}|<\frac{1}{k}, \text{ for all } j\geq j_k 
\end{equation}
Then define the following unitary diagonal matrix $e^D$ in terms of its $j,j$ entries (and zero elsewhere) whose construction is based in the $e^{D_{n_k}}$, and the corresponding $j_k$ mentioned above:
\begin{equation}\label{def: definicion de exp(D) que aproxima a diag(u0)}
(e^D)_{j,j}=\left\{\begin{array}{ll}
             1 , & \text{ if } 1\leq j<j_1\\
             \left(e^{D_{n_1}}\right)_{j,j} , & \text{ if } j_1\leq j<j_2\\
             \left(e^{D_{n_2}}\right)_{j,j} , & \text{ if } j_2\leq j<j_3\\
             \cdots & \cdots\\
             \left(e^{D_{n_k}}\right)_{j,j} , & \text{ if } j_k\leq j<j_{k+1}\\
             \cdots & \cdots
            \end{array}\right. .
\end{equation}
$D$ can be chosen as the anti-Hermitian diagonal matrix formed with the corresponding parts of $0$, $D_{n_1}$, $D_{n_2}$, $\dots$, $D_{n_k}$, \dots.

If we define $j_0=1$ and take any $j\in\N$, then equation  \eqref{eq: cosas que cumplen nk y jk} and definition \eqref{def: definicion de exp(D) que aproxima a diag(u0)} imply that
$$
|(\diag(u_0)-e^{D})_{j,j}|=|(\diag(u_0)-e^{D_{n_k}})_{j,j}|<\frac{1}{k}, \ \   \text{ provided } j_k\leq j<j_{k+1}
$$
Then $(\diag(u_0)-e^{D})_{j,j}\to 0$ as $j\to\infty$, and therefore $\diag(u_0)-e^{D}\in\kh$ since it is a diagonal matrix.

Using that also $u_0-\diag(u_0)\in\kh$ (see \eqref{eq: u0 - diag u0 es compacto}) we conclude that $\displaystyle\left(u_0-\diag(u_0)\right)+\left(\diag(u_0)-e^{D}\right)= u_0- e^D \in \kh$, which implies that $u_0\in\ukd$ and therefore $\ukd$ is closed.
\end{proof}

\begin{lem}\label{lema: epsilon donde vale e^K e^D=e^K+D}
 There exists $ \varepsilon_0>0$ and 
  $\delta_0>0$ such that if $\|u-1\|<\varepsilon_0$ and $u\in\ukd$ then there 
  exists $K, K'\in \kah$, $D\in \dah$ with $\|K\|, \|K'\|, \|K'+D\| <\delta_0$ 
  and $\|K'\|<2 \delta_0$, such that 
\begin{enumerate}
 \item [a)] $u=e^K e^D=e^{K'+D}$ with $K, D\in \exp^{-1}\left( B(1,3 
 \varepsilon_0)\right)$,
\item [b)] $K, K'$ and $ D$ are sufficiently close to $0$ and
\item [c)]  
$(K'+D)\in\exp^{-1}\left(B(1,\varepsilon_0)\right)\cap
 \bah$.
\end{enumerate}

\end{lem}
\begin{proof}
Let us fix $\delta_0>0$ such that fulfills two conditions:
\begin{enumerate}
	\item One of them is that if $V\in B(0,\delta_0)\cap \bah$ then $V$ is 
	sufficiently close to $0$ as in Definition \ref{def: suficientemente cerca 
	de cero}.
	\item The other one is that 
	$\exp:B(0,\delta_0)\cap \bah \to \exp\left(B(0,\delta_0)\right) \cap\uh$ is 
	a diffeomorfism considering the usual operator norm.
\end{enumerate}   The last requirement can be fulfilled after applying the 
inverse map theorem for Banach spaces. 

Then define $\varepsilon_0=\varepsilon>0$ such that  
\begin{equation}
 \label{def: el epsilon del lema}
B(1,\varepsilon)\subset B(1,3\varepsilon) \subset \exp\left(B(0,\delta_0) \right).
\end{equation}
If we take $u\in\ukd\cap B(1,\varepsilon)$, then there exists $K_1\in\kah$ and $D_1\in\dah$ such that $u=e^{K_1}e^{D_1}$. Observe that the $j,j$ entries of the diagonal of $u=e^{K_1}e^{D_1}$ are $e^{K_1}_{j,j} e^{D_1}_{j,j}$.

Then $\|u-1\|<\varepsilon$ implies that $|e^{K_1}_{j,j} e^{D_1}_{j,j}-1|<\varepsilon$ for all $j\in\N$. Suppose that $|e^{D_1}_{j,j}-1|\geq 2\varepsilon$ for infinite $j\in\N$. Then, using that $|e^{D_1}_{j,j}|=1$ we obtain that $|e^{-D_1}_{j,j}-1|=\left|e^{D_1}_{j,j}\left(e^{-D_1}_{j,j}-1\right)\right|=|1-e^{D_1}_{j,j}|\geq 2 \varepsilon$, and that $|e^{K_1}_{j,j}-e^{-D_1}_{j,j}|=\left|\left(e^{K_1}_{j,j}-e^{-D_1}_{j,j}\right)e^{D_1}_{j,j}\right|=|e^{K_1}_{j,j} e^{D_1}_{j,j}-1|<\varepsilon$ for infinite $j\in\N$.
Therefore, there must exist infinite $j\in\N$ such that
\begin{equation}\label{eq: desigualdades con menor que eps y mayor que 2eps}
|e^{K_1}_{j,j}-e^{-D_1}_{j,j}|<\varepsilon \text{  and  }
|e^{-D_1}_{j,j}-1|\geq 2\varepsilon
\end{equation}
(see Figure \ref{fig:es_o_no_de_Lie}).
\begin{figure}[h]
\includegraphics[width=.5\textwidth]{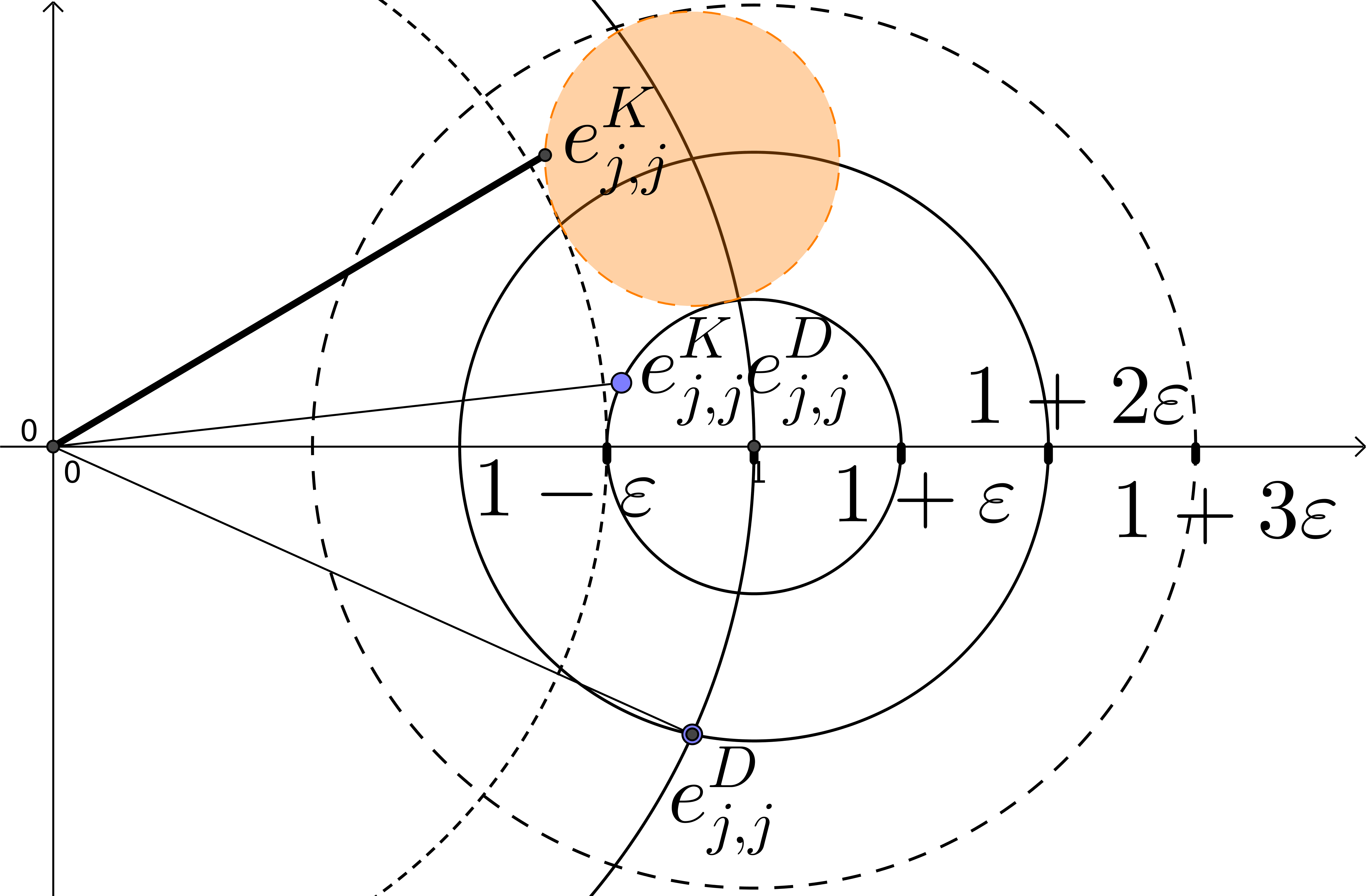}
	\caption{}\label{fig:es_o_no_de_Lie}
\end{figure}

Therefore
\begin{equation*}
\begin{split}
|e^{K_1}_{j,j}-1|&=
\left| e^{K_1}_{j,j}-e^{-D_1}_{j,j}+e^{-D_1}_{j,j}-1\right|
\geq
\left| |e^{K_1}_{j,j}-e^{-D_1}_{j,j}|-|e^{-D_1}_{j,j}-1|\right|\\
&=|e^{-D_1}_{j,j}-1|-|e^{K_1}_{j,j}-e^{-D_1}_{j,j}|\geq
2\varepsilon - \varepsilon=\varepsilon.
\end{split}
\end{equation*}
for infinite $j\in\N$, where we used \eqref{eq: desigualdades con menor que eps y mayor que 2eps} in the last equality and inequality. This is a contradiction because $e^{K_1}\in\uc$ and then $e^{K_1}-1\in\kh$ which implies that the diagonal of $e^{K_1}$ tends to $1$. Then $|e^{D_1}_{j,j}-1|\geq 2\varepsilon$ only for finite $j\in\N$. Choosing appropriately a compact anti-Hermitian diagonal $d$ we can construct $D_2=D_1+d\in\dah$ and $K_2\in\kah$ such that $u=e^{K_1}e^{D_1}=e^{K_2}e^{D_2}$ and  $|e^{D_2}_{j,j}-1|\leq 2\varepsilon$ for all $j\in\N$ (see Proposition \ref{prop: diagonales compactas d para e^-d e^d entre e^K y e^D}).
Then, $\|e^{D_2}-1\|<2\varepsilon$ and $\|e^{D_2}-1\|=\|e^{-D_2}\left(e^{D_2}-1\right)\|=\|1-e^{-D_2}1\|=\|e^{-D_2}-1\|<2\varepsilon$. Moreover, since $e^{D_2}$ is unitary, 
 \begin{equation}
 \label{eq: cota de e^K_2 - 1<3 epsilon}
 \begin{split}
\| e^{K_2}-1\|\leq&\| e^{K_2}-e^{-D_2}\|+\|e^{-D_2}-1\|
\\=&\left\| \left(e^{K_2}-e^{-D_2}\right)e^{D_2}\right\|+\|e^{-D_2}-1\|\\
=&\left\|e^{K_2}e^{D_2}-1\right\|+\|e^{-D_2}-1\|
\\=& \left\|u-1\right\|+\|e^{-D_2}-1\| <\varepsilon+2\varepsilon  =3 \varepsilon
 \end{split}
\end{equation}
We obtained that $\|e^{D_2}-1\|<2\varepsilon$ and $\|e^{K_2}-1\|<3\varepsilon$ which implies that $e^{D_2}$ and $e^{K_2}\in\exp\left(B(0,\delta_0)\right)$ (see the definition of $\varepsilon$ in \eqref{def: el epsilon del lema}. 
Therefore, using that 
$\exp: B(0,\delta_0)\cap \bah \to \exp\left(B(0,\delta_0) \right)\cap \uh$ is a 
diffeomorfism, there exist unique $D$ and $K$ in 
$\exp^{-1}\left(B(0,3\varepsilon)\right)\cap\bah\subset B(0,\delta_0)\cap\bah$ 
such that $e^D=e^{D_2}$ and $e^K=e^{K_2}$. Standard calculations can show that 
under these conditions $D$ must be diagonal and $K$ compact. Hence $D\in 
\exp^{-1}\left(B(0,3\varepsilon)\right)\cap\dah $ and $K\in 
\exp^{-1}\left(B(0,3\varepsilon)\right)\cap\kah$. Moreover, since $D,K \in 
B(0,\delta_0)$ they are sufficiently close to $0$. Then using \eqref{eq: e^K_1 
e^D_1 = e^D_1 e^K'=e^K''+D_1}  there exists $K'\in\kah$ such 
that
$$
u=e^K e^D=e^{K'+D} \in \uda.
$$
Observe that $K, D\in \exp^{-1}\left( 
B(0,3 \varepsilon)\right)$ and $K, D$ are sufficiently close to $0$ as required 
in 
items a) and b) of the lemma. Moreover, since $u$ 
belongs to the neighborhood $\ukd\cap B(1,\varepsilon)\subset 
\exp(B(0,\delta_0))$ where the exponential is a diffeomorphism, the exponent 
$K'+D$ is unique in $\exp^{-1}( B(1,\varepsilon) ) \subset B(0,\delta_0)\cap 
\bah$ (which proves item c) ), and therefore 
$\|K'+D\|<\delta_0$. This in turn implies that 
$$
\|K'\|\leq 
\|K'+D\|+\|-D\|<2\delta_0
.
$$

Note that if we originally considered $\delta_0$ even 
smaller in our original assumptions (for example, modifying our requirement (1) 
taking $\delta_0$ such 
that if $V\in 
B(0,2\delta_0)\cap \bah$ then $V$ is 
sufficiently close to $0$), then this last inequality would 
imply that $K'$ 
would be 
also sufficiently 
close to $0$ which ends the proof of item b). 
%
\end{proof}
\begin{prop}\label{prop: exp es difeo local}
 There exists $\mathcal{V}\subset \bah$ an open neighborhood of $0$ such that 
$$ 
\exp\left(\mathcal{V}\cap \left(\kah + \dah\right) \right) 
=
\exp\left(\mathcal{V}\right) \cap \ukd
$$
\end{prop}
\begin{proof}
Take $\mathcal{V}=\exp^{-1}\left(B(1,\varepsilon_0)\right)\cap \bah$, where 
$\varepsilon_0$ is the one from Lemma \ref{lema: epsilon donde vale e^K 
e^D=e^K+D}. Then, as seen in that lemma, for every $u\in\ukd$, and 
$u=e^{K_1}e^{D_1}\in B(1,\varepsilon_0)$ with $K_1\in\kah$ and $D_1\in\dah$, 
there exist $K\in\kah\cap \exp^{-1}\left(B(1,3\varepsilon_0)\right)$, $
K'\in \kah$ and $D\in\dah\cap 
 \exp^{-1}\left(B(1,3\varepsilon_0)\right)$ such that
\begin{equation}
 \label{eq: u es e^k+d}
u=e^K e^D=e^{K'+D} ,  \ \text{ with } 
(K+D)\in\exp^{-1}\left(B(1,\varepsilon_0)\right)\cap \bah.
\end{equation}
 Suppose first that 
$V\in\mathcal{V}$ and $e^V\in\exp\left(\mathcal{V}\right) \cap \ukd$. Then, as commented in \eqref{eq: u es e^k+d}, there exists $K\in\kah$ and $D\in \dah$ such that $e^V=e^{K+D}$. Since the exponential is a diffeomorfism restricted to the neighborhood $\mathcal{V}$ then $V=K+D$. Therefore 
$$ 
\exp\left(\mathcal{V}\right) \cap \ukd
\subset
\exp\left(\mathcal{V}\cap \left(\kah + \dah\right) \right).
$$

Now suppose that $V\in\mathcal{V}$ and $$
e^V=e^{K+D}\in 
\exp\left(\mathcal{V}\cap \left(\kah + \dah\right) \right)$$ 
with $K\in\kah$ and $D\in\dah$. Then clearly $e^{V}\in \exp\left(\mathcal{V}\right)$ and using (3) from Properties \ref{propiedades} we obtain that also $e^V=e^{K+D}\in\ukd$ holds. This proves that $ 
\exp\left(\mathcal{V}\cap \left(\kah + \dah\right) \right) 
\subset
\exp\left(\mathcal{V}\right) \cap \ukd
$ which concludes the proof.

\end{proof}

\begin{prop}\label{prop: el algebra de Lie de ukd es kah+dah}
 $\{ X\in\bah: e^{t X}\in \ukd , \forall t\in\R\}=\kah+\dah$.
\end{prop}
\begin{proof}
The property (3) of Proposition \ref{propiedades} directly implies that $e^{t(K+D)}=e^{tK+tD}\in \ukd$ for all $t\in\R$ and therefore $\kah+\dah\subset L(\ukd)$.

Suppose now that $X\neq 0$ ($0$ is a trivial case) and let $X\in L(\ukd)$, then $e^{t X}\in \ukd , \forall t\in\R$. In particular $e^{t X}\in \ukd$ holds for small $|t|$, for example for $t_0=\frac{\delta_0}{2\|X\|}<\frac{\delta_0}{\|X\|}$ where $\delta_0>0$ is the constant used in the proof of Lemma \ref{lema: epsilon donde vale e^K e^D=e^K+D}. Then, $\|t_0 X\|=|t_0| \|X\|<\delta_0$ and $u=e^{t_0 X}\in\ukd$. Therefore using Lemma \ref{lema: epsilon donde vale e^K e^D=e^K+D}, there exists $K\in\kah$ and $D\in\dah$ such that 
$$
e^{t_0 X} = e^{K+D}.
$$
The constant $\delta_0$ of the proof of Lemma \ref{lema: epsilon donde vale e^K e^D=e^K+D} is chosen such that $\exp:B(0,\delta_0)\cap \bah \to \exp\left(B(0,\delta_0)\right) \cap\uh$ is a diffeomorfism. Then $t_0 X=K+D$ and therefore $X=1/t_0 (K+D)=1/t_0 K+1/t_0 D\in \kah+\dah$ as required.
\end{proof}

\begin{rem}\label{def subgrupo de Lie}
Following V.2.3 \cite{neeb} and  \cite{neeb_pianzola} (page 428) we call $H$ a Lie subgroup of $G$ if $H$ is a closed subgroup of a Banach--Lie group $G$ which is itself a Lie group relative to the induced topology.
\end{rem}

Therefore the previous results allow us to state the following.

\begin{teo}\label{teo: ukd es subgrupo de Lie}
$\ukd$ is a Lie subgroup of $\uu(\h)$ and its Lie algebra is $L(\ukd)=\kah+\dah$.
\end{teo}
\begin{proof}
 According to the definition of Lie subgroup mentioned in the previous Remark \ref{def subgrupo de Lie}, Theorem \ref{teo: ukd es cerrado} and Proposition \ref{prop: exp es difeo local} imply that $\ukd$ is a Lie subgroup of $\uu(\h)$. 

The equality $L(\ukd)=\kah+\dah$ follows from Corollary V.2.2 in \cite{neeb} and Proposition \ref{prop: el algebra de Lie de ukd es kah+dah}.
\end{proof}

Although there exist stronger notions of Lie subgroups those cannot be used for $\ukd$ since its Lie algebra $\kah + \dah$ is not complemented in $\bah$ (the Lie algebra of $\uh$).

\begin{rem}\label{rem: generalizacion del subgrupo ukd} \textbf{Generalization of the $\ukd$ group.}
	
The  proofs of some of the basic properties we use in the study of $\ukd$ require that the exponential $\exp: \kah\to \uc$ must be surjective. This is the reason why the following generalization involves ideals $\J$ with this property.
	
If $\J\subset \bb(\h)$ is either of the two-sided closed ideals of $p-$Schatten operators (for $p\in[0,\infty)$) or $\kk(\h)$, and $\aaa$ is any  C$^*$ subalgebra of $\bb(\h)$, then the following unitary sets of $\uu(\h)$ can be defined, by analogy with \eqref{def: subconjuntos de operadores unitarios}:

\begin{equation}
\begin{split}
\uu_{\mathcal{J}}&=\{u\in\uu(\h): u-1\in\J\},\\
\uu_{\mathcal{J},\mathcal{A}} &=\{u\in \uu(\h):\exists\ A\in \aaa^{ah}\ {\text{ such\ that }}\ u-e^A\in \J\}, \\
\uu_{\mathcal{A}} &=\{u\in \uu(\h):\exists\ A\in \aaa^{ah}\ {\text{ such\ that }}\ u=e^A\},\\
\uu_{\mathcal{J}+\mathcal{A}}&=\{u\in \uu(\h):\exists\ J\in \J^{ah}\ {\text{ and }}\ A\in \aaa^{ah}\ {\text{ such\ that }}\ u=e^{J+A}\}.\nonumber
\end{split}
\end{equation}
The groups $\uu_{\mathcal{J}}$, where $\J$ is any $p-$Schatten ideal of $\bb(\h)$, were studied in \cite{andruchow_larotonda_recht_Finsler_geometry_p_Schatten}. 

It can be proved that the previous sets of unitary operators satisfy the following properties:
\be 
\item $\uu_{\mathcal{J}}=\{e^J\in\uu(\h): J\in\J^{ah},\left\| J\right\|\leq\pi \}$.
\item $\uu_{\mathcal{J},\mathcal{A}}$ is a group, equals 
\begin{equation*}
\begin{split}
\uu_{\mathcal{J},\mathcal{A}}=\left\{u\in \uu(\h):\right. &\exists\ J\in \J^{ah}, \left\|J\right\|\leq \pi,\ {\rm and}\ 
\\ & \left. A\in \aaa^{ah}\ {\rm such\ that}\ u=e^{J}e^A\right\}.
\end{split}
\end{equation*} 
and  $\uu_{\mathcal{J},\mathcal{A}}=\uu_{\mathcal{J}}\ \uu_{\mathcal{A}}=\uu_{\mathcal{A}}\ \uu_{\mathcal{J}}.$
\item $\uu_{\mathcal{A}}\subsetneq \uu_{\mathcal{J},\mathcal{A}}$.
\item $\uu_{\mathcal{J}}\subsetneq \uu_{\mathcal{J},\mathcal{A}}$.
\item $\uu_{\mathcal{J}+\mathcal{A}}\subseteq \uu_{\mathcal{J},\mathcal{A}}$.
\item If $u\in \uu_{\mathcal{J},\mathcal{A}}$ then $u=J'+A'$, with $J'\in \J$ and $A'\in \uu_{\mathcal{A}}$.
\item For every $J\in\J^{ah}$ and $A\in \aaa^{ah}$ there
exists $J'\in {\mathcal{J}^{ah}}$ such that $e^Je^A=e^Ae^{J'}$.
\item $\uu_{\mathcal{J}}\subsetneq \uu_{\mathcal{J}+\mathcal{A}}$ 
\item If $\aaa^{ah}\cap \J\neq \emptyset$, then 
$$\uu_{\mathcal{J}}=\{u\in \uu(\h):\exists\ A\in \aaa^{ah}\cap \J\ {\rm such\ that}\ u-e^A\in \J\}.$$
\item For every $J\in \J^{ah}$ and $A\in \aaa^{ah}$ sufficiently close to $0$, there exist $J''\in \J^{ah}$ and $A'\in \aaa^{ah}$ such that
$$e^{J}e^A=e^{J''+A'}.$$
\item  For every $J_1,J_2\in\J^{ah}$ and $A_1,A_2\in \aaa^{ah}$ the following statements are equivalent:
\bit
\item $e^{J_1} e^{A_1}=e^{J_2} e^{A_2}$.
\item There exists $a\in \aaa^{ah}\cap \J$ such that $e^{J_2}=e^{J_1}e^{-a}$ and $e^{A_2}=e^a e^{A_1}=e^{a+A_1}$.
\eit
\ee
Property $(1)$ has been proved for $p-$Schatten ideals in \cite{andruchow_larotonda_recht_Finsler_geometry_p_Schatten} (Remark 3.1) and for $\kk(\h)$ see Proposition \ref{prop: uc es ek}.
Properties $(2)$-$(9)$ and $(11)$ may be proved in much the same way as Propositions \ref{prop: Ukd es un grupo y es Uk.Ud}, \ref{propiedades} and \ref{prop: diagonales compactas d para e^-d e^d entre e^K y e^D}. Property $(10)$ involves the \bch\ series expansion $\log\left(e^{J}e^A\right)$ for the Lie algebra $\bah$ (see the Preliminaries).
\end{rem}

\section{Minimal length curves in the orbit of a compact self-adjoint operator} 
 
Consider the unitary Fredholm orbit of a compact operator $$
b=\diag\left(\{\lambda_i\}_{i\in \N}\right)\in \D(\kh^h)
$$ 
with  $\lambda_i\neq\lambda_j$ for each $i\neq j$, and the orbit 
$$
\mathcal{O}_b=\{ubu^*:u\in\mathcal{U}_k)\}.
$$
The isotropy subgroup of $c=e^{K_0}be^{-K_0}\in\ob$, with $K_0\in\kah$ for the action $L_u\cdot c=u c u^*$, with $u\in\uc$, is $\I_c=\{e^{K_0}e^d e^{-K_0}: d\in \D(\kah)\}=\{e^{e^{K_0}d e^{-K_0}}: d\in \D(\kah)\}$. $(T\ob)_c$ can be identified with the quotient space $\kah/\D(\kah)$ for every $c\in\ob$. The projection to the quotient $\kah/\D(\kah)$ defines a Finsler metric as
$$\left\|x\right\|_{e^{K_0}b
e^{-K_0}}=\left\|\left[Y\right]\right\|=\inf_{D\in \D(\kah)}\ \left\|Y+e^{K_0}D 
e^{-K_0}\right\|$$
for each class $[Y]=\left\{Y+e^{K_0}D e^{-K_0}:D\in \D(\kah)\right\}$ \ and 
$x=Yc-cY\in(T\ob)_c$, for $c=e^{K_0}be^{-K_0}\in\ob$.  
This metric is invariant under the action of $L_{e^K}$ for $e^k\in\uc$ (see \cite{bottazzi_varela_DGA}). This invariance implies that the curve $\gamma:[a,b]\to\ob$, such that $\gamma(0)=b$ and $\dot{\gamma}(0)=x$ has the same length than $\beta(t)=L_{e^K}\cdot \gamma:[a,b]\to\ob$ and satisfies $\beta(0)=L_{e^K}\cdot b$,  $\dot\beta(0)=L_{e^K}\cdot \dot{\gamma}(0)=L_{e^K}\cdot x$, for $e^K\in\uc$. 

\begin{prop} \label{prop: equiv}
Let $b=\diag\left(\{\lambda_i\}_{i\in \N}\right)\in \D(\kh)$ with $\lambda_i\neq\lambda_j$ for each $i\neq j$ and $Z_0=S_0+D_0$, with $S_0\in \kah$ and $D_0\in \D(\bah)$. Then $e^{Z_0}be^{-Z_0}\in \ob$.
\end{prop}

\begin{proof}
Observe that $D_0$ is not necessarily compact. 
Using Remark \ref{obs2} the exponential $e^{Z_0}=e^{S_0+D_0}$ can be rewritten as
$$e^{Z_0}=e^{D_0}+S_0 \Psi(S_0,D_0).$$
Then, $\left(e^{D_0}+S_0 \Psi(S_0,D_0)\right)e^{-D_0}$ is unitary and
$S_0 \Psi(S_0,D_0)e^{-D_0}=e^{D_0-D_0}-1+S_0 \Psi(S_0,D_0)e^{-D_0}\in \kk(\h)$ since $S_0\in\kk(\h)$. Moreover
$$
\left(e^{D_0}+S_0 \Psi(S_0,D_0)\right)e^{-D_0}-1\ \in \kk(\h),
$$ 
which implies that $e^{S_0+D_0}e^{-D_0}-1\in \kk(\h)$. Therefore, by Proposition 3 in \cite{bottazzi_varela_DGA} there exists $K\in \kah$ such that
$$e^{S_0+D_0}e^{-D_0}=e^K\text{ , and therefore } e^{S_0+D_0}=e^Ke^{D_0}.$$
Then
$$e^{Z_0}be^{-Z_0}=e^{S_0+D_0}be^{-S_0-D_0}=e^Ke^{D_0}be^{-D_0}e^{-K}=e^Kbe^{-K}\in \ob.$$
\end{proof}
\begin{teo}\label{teo: exp(z0)b exp(z0) es corta en Ob}
Let $Z_0=S_0+D_0$ such that $S_0\in \kah$, $D_0\in \D(\bah)$ and $b=\diag \left(\{\lambda_i\}_{i\in \N}\right)\in \D(\kh)$ with $\lambda_i\neq\lambda_j$ for each $i\neq j$, and $\gamma(t)=e^{tZ_0}be^{-tZ_0},\ \forall\ t\in \R$.
 Then,
\subitem a) 
$\gamma(t)\in \ob,\ \forall\ t\in \R$, and

\subitem b) 
if $Z_0$ is minimal for $\D(\bah)$ (see \eqref{def: operador minimal autoajunto} and the Preliminaries) then $\gamma: \left[-\frac{\pi}{2\left\|Z_0\right\|},\frac{\pi}{2\left\|Z_0\right\|}\right]\to\ob$ is a minimal length curve on $\ob$ considering the distance \eqref{def: distancia rectificable}.
\end{teo}

\begin{proof}
The assertion of item a) follows directly from Proposition \ref{prop: equiv}. Note that the a) holds even though $Z_0$ may not be compact.

In order to prove b) consider $\PP_b=\{ubu^*:\ u\in\uu(\h)\}$, then by Theorem II in \cite{dmr1}, since $Z_0$ is minimal, the curve $\gamma$ has minimal length over all the smooth curves in $\PP_b$ that join $\gamma(0)=b$ and $\gamma(t)$, with $\left|t\right|\leq \frac{\pi}{2\left\|Z_0\right\|}$. Since clearly $\ob\subseteq\PP_b$, then for each $t_0\in \left[-\frac{\pi}{2\left\|Z_0\right\|},\frac{\pi}{2\left\|Z_0\right\|}\right]$ follows that $\gamma$ is minimal in $\ob$, that is
$$\longi(\gamma)=\dist(b,\gamma(t_0)),$$
where $\dist(b,\gamma(t_0))$ is the rectifiable distance between $b$ and $\gamma(t_0)$ defined in \eqref{def: distancia rectificable} of the Preliminaries.
\end{proof}
\begin{rem}\label{rem: siempre hay minimal acotado}
Recall that for every $S_0$ there always exists a minimal $Z_0\in\bah$ as that mentioned in Theorem \ref{teo: exp(z0)b exp(z0) es corta en Ob} although it may not be compact (see \cite{dmr1},  \cite{bottazzi_varela_LAA}).
\end{rem}
\begin{prop} \label{orbitas iguales}
If $b=\diag \left(\{\lambda_i\}_{i\in \N}\right)\in \D(\kh)$ with $\lambda_i\neq\lambda_j$ for each $i\neq j$, then 
$\ob=\oo^{\ukd}_b$ and $\oo^{\uda}_b\subseteq \oo_b$. 
\end{prop}

\begin{proof}
 Since $b$ is diagonal and $e^Ke^Dbe^{-D}e^{-K}=e^Kbe^{-K}$ follows that $\ob=\oo^{\ukd}_b$. The inclusion $\oo^{\uda}_b\subseteq \oo_b$ is trivial because $b$ is diagonal and $\ukd\supset \uda$ (see (3) in Proposition \ref{propiedades}).
\end{proof}

\begin{rem}\label{rem: ukd y uk misma metrica de Finsler en Ob}
Under the same assumptions of Proposition \ref{orbitas iguales}, if $c\in\ob$ the following identifications can be made
$$
(T\ob)_c\cong (T\uc)_1/(T\I_b)_1 = \kah/\D(\kah)\\ 
$$
and
\begin{equation*}
\begin{split}
	(T\ob)^{\ukd}_c&\cong (T\ukd)_1/(T\I_b)_1\\
	&\cong \left( \kah+\dah\right)/\dah.
\end{split}
\end{equation*} 
Moreover the norm on each quotient coincides on every class since for $K\in\kah$ holds  $\|[K]\|=\inf_{d\in\D(\kah)}\|K+d\|=\inf_{D\in\dah}\|K+D\|$ (see for example Proposition 5 in \cite{bottazzi_varela_LAA}). Therefore the Finsler metrics defined by the subgroups $\uc$ and $\ukd$ coincide on $\ob$.
\end{rem}

Let $c=L_{e^{K_0}}\cdot b=e^{K_0}be^{-K_0}\in\ob$ (for $K_0\in\kah$) and $x\in T(\ob)_c$. Then  there always exists a vector $z_c=L_{e^{K_0}}\cdot	Z_0$, with $ z_c c-z_c c=x$ minimal for $\{F\in \bah: F c-c F=0\}=\{F\in \bah: F=e^{K_0} D e^{-K_0} , \text{ for } D\in\dah \}$ such that $Z_0$ is minimal for $\dah$ as in Theorem \ref{teo: exp(z0)b exp(z0) es corta en Ob} b) and Remark \ref{rem: siempre hay minimal acotado}. 
That is,
\begin{equation*}
\begin{split}
\|x\|_c&=\|[z_c]\|_c=\inf_{F\in \bah\cap\{c\}'}\|z_c+F\|
\\
&=\inf_{D\in \dah}\|e^{K_0}Z_0e^{-K_0}+e^{K_0}De^{-K_0}\|\\
&=\inf_{D\in \dah}\|Z_0+D\|
=\|[Z_0]\|
\end{split}
\end{equation*}
This equality and the left invariance of the action $L_{e^K}$ imply that the curve
$$
\beta(t)= e^{t z_c} c e^{-t z_c}
$$
for $t\in[-\frac{\pi}{2\|z_c\|}, \frac{\pi}{2\|z_c\|}]$, 
satisfies $\beta(0)=c$, $\dot\beta(0)=x=z_c c-c z_c$ and  $\longi(\beta|_{[a,b]})=\longi(\gamma|_{[a,b]})$ for the curve $\gamma$ mentioned in Theorem \ref{teo: exp(z0)b exp(z0) es corta en Ob} and every $[a,b]\subset [-\frac{\pi}{2\|z_c\|}, \frac{\pi}{2\|z_c\|}]$.
The previous comments and results allow us to prove the following.

\begin{cor}\label{cor: curvas minimales con comienzo en c}
	Let $b=\diag \left(\{\lambda_i\}_{i\in \N}\right)\in \D(\kh^{h})$, 
	$\lambda_i\neq \lambda_j$ for each $i\neq j$, 
	$c=e^{K_0}be^{-K_0}\in \ob$, with $K_0\in\kah$, and $x\in T(\ob)_c$. Then 
	there exists $Z_0\in\bah$ minimal for $\dah$ such that 
$\beta(t)=e^{t z_c} c e^{-tz_c}\in \ob$
 for all $ t\in \R
$, $z_c=e^{K_0} Z_0 e^{-K_0}$ and $x=L_{e^{K_0}}\cdot (Z_0 b-bZ_0)$. Moreover, $\beta:[-\frac{\pi}{2\|z_c\|}, \frac{\pi}{2\|z_c\|}]\to \ob$ is a minimal length curve in $\ob$ such that $\beta(0)=c$, $\dot\beta(0)=x$  considering the distance \eqref{def: distancia rectificable}.
\end{cor}
\begin{proof}
	Given $x\in T(\ob)_c$ we can choose $Z_0\in\kah$, such that $Z_0 b-b Z_0=L_{e^{-K_0}}\cdot x\in T(\ob)_b$ and that satisfies $\|Z_0\|=\inf_{D\in\dah}\|Z_0+D\|$ as in Theorem  \ref{teo: exp(z0)b exp(z0) es corta en Ob} and Remark \ref{rem: siempre hay minimal acotado}. $Z_0$ is minimal for $\dah$ and therefore Theorem \ref{teo: exp(z0)b exp(z0) es corta en Ob} b) applies and $\gamma(t)=e^{tZ_0}be^{-tZ_0}$ is a short curve for $t\in[-\frac{\pi}{2\|Z_0\|}, \frac{\pi}{2\|Z_0\|}]$. 
	Direct calculations show that $x=L_{e^{K_0}}\cdot (Z_0 b-b Z_0)=(L_{e^{K_0}}\cdot Z_0)( L_{e^{K_0}}\cdot b)-(L_{e^{K_0}}\cdot b) ( L_{e^{K_0}}\cdot Z_0)$.
If $z_c=L_{e^{K_0}}\cdot Z_0=e^{K_0}Z_0e^{-K_0}$ it is apparent that if $\beta(t)=e^{tz_c}ce^{-tz_c}$, for $t\in\R$ and $c=L_{e^K_0}\cdot b$, then $\beta(0)=c$ and $\dot{\beta}(0)=z_c c-c z_c=L_{e^{K_0}}\cdot (Z_0 b-b Z_0)=x$

Similar considerations as those in Proposition \ref{prop: equiv} using that $$\beta(t)=e^{K_0}e^{tZ_0}e^{-K_0}e^{K_0}be^{K_0}e^{-K_0}e^{-tZ_0}e^{-K_0}=L_{e^{K_0}} \left(e^{tZ_0}be^{-tZ_0}\right)=L_{e^{K_0}} \left(\gamma(t)\right)
	$$ 
	imply that $\beta(t)\in\ob$ for all $t\in\R$. 
	
	Standard arguments of homogeneous spaces (invariance of the Finsler metric) imply that $\beta$ is a curve of minimal  length when is defined in the interval $[-\frac{\pi}{2\|z_c\|}, \frac{\pi}{2\|z_c\|}]  =[-\frac{\pi}{2\|Z_0\|}, \frac{\pi}{2\|Z_0\|}]$ as $\gamma$ is.
\end{proof}

\begin{rem}
	Theorem \ref{teo: exp(z0)b exp(z0) es corta en Ob}, Remark \ref{rem: ukd y uk misma metrica de Finsler en Ob} and Corollary 	\ref{cor: curvas minimales con comienzo en c}  allow us to describe short curves $\beta$ in $\ob$ with initial condition $\beta(0)=c$ even for velocity vectors $x\in T(\ob)_c$ that do not have a minimal compact lifting $Z_0$. Thus $\uc$ is an example of a group whose action on $\ob$ has short curves that need not to be described necessarily with minimal vectors $F$ that belong to $\{F\in \bah: F c-c F=0\}$. Nevertheless there exists another group $\ukd$ acting on $\ob$ such that its $b$-orbit coincides with that of $\uc$, defines the same Finsler metric on it and where every short curve can be described by means of a minimal lifting.
\end{rem}
The previous geometric properties allow the following results relating the quotient norm $\displaystyle\|[K]\|=\inf_{D\in\dah}\|K+D\|$ of two anti-Hermitian compact operators.
\begin{prop}
Let $K_1,K_2\in\kah$ and $D_1,D_2\in \D(\bah)$ such that $e^{tK_1}e^{D_1}=e^{tK_2}e^{D_2}$ for all $t\in[0,1]$. Then,
$$\left\|[K_1]\right\|=\left\|[K_2]\right\|.$$ 
\end{prop}
\begin{proof}
Let $b=\diag \left(\{\lambda_i\}_{i\in \N}\right)\in \D(\kh^h)$ with $\lambda_i\neq\lambda_j$ for each $i\neq j$. The equality $e^{tK_1}e^{D_1}=e^{tK_2}e^{D_2}$ implies that
$$e^{tK_1}be^{-tK_1}=e^{tK_2}be^{-tK_2},$$
for all $t\in[0,1]$. If we consider $\alpha,\s: [0,1]\to \ob$, defined by
$$\s(t)=e^{tK_1}be^{-tK_1}\ {\rm and}\ \alpha(t)=e^{tK_2}be^{-tK_2},$$
then
$$\longi(\s)=\longi(\alpha)\Rightarrow\ \int_0^1\left\|\s'(t)\right\|_{\s(t)}dt=\longi(\alpha)=\int_0^1\left\|\alpha'(t)\right\|_{\alpha(t)}dt$$
$$\Rightarrow\ \left\|[K_1]\right\|=\left\|[K_2]\right\|.$$
This concludes the proof.
\end{proof}

\begin{prop}
Let $K\in\kah$ and $D\in \D(\bah)$ such that $K+D$ is minimal and $\left\|K+D\right\|<\frac{\pi}{2}$. Then, if $K'\in \kah$ is such that $e^{K+D}=e^{K'}e^D$ the inequality 
$$
\|K+D\|=\left\|[K]\right\|\leq\left\|[K']\right\|
$$
holds.
 
\end{prop}
\begin{proof}
Let $b=\diag \left(\{\lambda_i\}_{i\in \N}\right)\in \D(\bh)$ with $\lambda_i\neq\lambda_j$ for each $i\neq j$, and consider $\alpha,\s: [0,1]\to \ob$, defined by
$$\s(t)=e^{t(K+D)}be^{-t(K+D)}\ {\rm and}\ \alpha(t)=e^{tK'}be^{-tK'}.$$
Observe that since $e^{K+D}=e^{K'} e^{D}$ then $\beta(0)=\alpha(0)$ and $\beta(1)=\alpha(1)$.

But $\s$ has minimal length between all rectifiable unitary curves that join $b$ with $\s(1)=e^{K+D}be^{-(K+D)}=e^{K'}be^{-K'}=\alpha(1)$ (see Corollary \ref{cor: curvas minimales con comienzo en c} and \cite{dmr1}).

Therefore
$$\longi(\s)\leq \longi(\alpha)$$
$$
\Rightarrow 
\int_0^1\left\|\s'(t)\right\|_{\s(t)}dt=\left\|Kb-bK\right\|_b= 
\|K+D\|= \left\|[K]\right\|$$
$$\leq\int_0^1\left\|\alpha'(t)\right\|_{\alpha(t)}dt=\left\|K'b-bK'\right\|_b=\left\|[K']\right\|.$$
\end{proof}
 

\subsection*{Acknowledgments}
We acknowledge Professors Esteban Andruchow, Gabriel Larotonda and Eduardo Chiumiento for fruitful conversations and suggestions along the preparation of this manuscript.

\end{document}